\documentclass{amsart}
\usepackage{amsthm, amsfonts, amssymb, latexsym, color, tikz}
\usepackage{amsmath,amssymb,amsbsy,amsfonts,latexsym,amsopn,amstext, amsxtra,euscript,amscd,bm}
\usepackage{enumerate,enumitem}

\usepackage{hyperref}
\usepackage{lipsum}
\usepackage{fixme}
\fxsetup{
    status=draft,
    author=,
    layout=margin,
    theme=color
}

\newcommand{\E}{\mathsf{E}}
\newcommand{\F}{\mathbb{F}}
\newcommand{\Z}{\mathbb{Z}}

\usepackage{amsmath}
\usepackage[a4paper, total={6in, 8in}]{geometry}
\newcommand{\ST}{Szemer\'edi-Trotter~}

\newtheorem{proposition}{Proposition}
\newtheorem*{proposition*}{Proposition}

\newtheorem{lemma}{Lemma}
\newtheorem*{corollary*}{Corollary 5}

\newtheorem{theorem}{Theorem}

\title{Attaining the exponent 5/4 for the sum-product problem in finite fields}
\date\today
\author[A.\ Mohammadi]{Ali Mohammadi}

\address{A.M.: School of Mathematics, Institute for Research in Fundamental Sciences (IPM),
Tehran, Iran.}
\email{a.mohammadi@ipm.ir}

 \author[S.\ Stevens]{Sophie Stevens}
\address{S.S.: Johannn Radon Institute for Computational and Applied Mathematics (RICAM), Linz, Austria}
\email{sophie.stevens@oeaw.ac.at}

\begin{document}
\maketitle

\begin{abstract}
    We improve the exponent in the finite field sum-product problem from $11/9$ to $5/4$, improving the results of Rudnev, Shakan and Shkredov \cite{RudShaShk}. That is, we show that if $A\subset \mathbb{F}_p$ has cardinality $|A|\ll p^{1/2}$ then 
    \[
    \max\{|A\pm A|,|AA|\} \gtrsim |A|^\frac54
    \]
    and
    \[
    \max\{|A\pm A|,|A/A|\}\gtrsim |A|^\frac54\,.
    \]
\end{abstract}

\section{Introduction}
Throughout the paper, we use $\F$ to denote an arbitrary field, $p$ a prime and $\F_p$ the finite field of order $p$. Given sets $A, B\subseteq \F$, we define their sum set by $A+B:= \{a+b: a\in A, b\in B\}$, and similarly define difference, product and ratio sets. In the sum-product problem over fields, we seek to establish that for any $0<\varepsilon<1$ and finite subset $A\subseteq \F$ (with appropriate conditions) we have 
\begin{equation}\label{eqn:SPC}
\max\{|AA|,|A+A|\}\gg |A|^{1+\varepsilon}.
\end{equation}
This naturally extends a question of Erd\H{o}s and Szemer\'edi~\cite{ES} over $\Z$.
Over finite fields, the first non-trivial result was achieved by Bourgain, Katz and Tao~\cite{BKT}, under the necessary condition that $|A| = o(|\mathbb{F}|)$; statements of the form~\eqref{eqn:SPC} can hold for subsets of finite fields only if the given set is small enough. Notably, by a construction of Garaev~\cite{Gar08}, for any $N\leq p$ there exists a subset $A \subseteq \F_p$ with $|A| = N$ such that
\begin{equation}
\label{eqn:SPUB}
\max\{|A+A|, |AA|\} \ll p^{1/2}N^{1/2}.
\end{equation}
Garaev~\cite{Gar08} also proved the lower-bound
\begin{equation}
\label{eqn:SpGarLS}
      \max\{|A+A|, |AA|\} \gg \min\{|A|^{2}p^{-1/2}, |A|^{1/2}p^{1/2}\},
 \end{equation}
which, as \eqref{eqn:SPUB} shows, is sharp up to constants in the range $|A|>p^{2/3}$. However, this bound is trivial in the range $|A|\leq p^{1/2}$. See also \cite[Theorem~5]{GraSol} for an improvement of \eqref{eqn:SpGarLS} in the range $p^{1/2}<|A|\leq p^{5/8}$. 

For sets of size less than $p^{1/2}$, Garaev \cite{Gar} first quantified the sum-product estimate explicitly, based on the method of Bourgain, Katz and Tao \cite{BKT}. By refining the same method, this estimate was improved incrementally in the series of papers (\cite{KaSh, BG,Li}), culminating in the apparent limit of this approach of $\varepsilon = 1/11 - o(1)$ by Rudnev~\cite{Rud12}.
Using different ideas based on an incidence result of Rudnev~\cite{Rud18}, Roche-Newton, Rudnev and Shkredov~\cite{RNRuShk} improved the exponent to the value $\varepsilon = 1/5$. A noteworthy feature of this result is that it holds for subsets of arbitrary fields $\F$, and under the constraint $|A|< p^{5/8}$ if char$(\F) = p>2$. 

In the reals, Elekes~\cite{El} instigated the use of tools from incidence geometry in the study of the sum-product problem, specifically a result of Szemer\'edi and Trotter~\cite{SzTr} on the number of incidences between points and lines over the real plane; Elekes proved that ~\eqref{eqn:SPC} holds with $\varepsilon = 1/4$ over the reals.
We match Elekes' bound in this paper, showing that 
it is actually the beautiful geometric ideas in Solymosi's argument \cite{Solymosi}, enabling $\varepsilon = 1/3$, that distinguish the improvements in the reals from those in finite fields. See also Solymosi \cite{Solymosi05}.
In the reals, the current best-known exponent $\varepsilon = 1/3 + 2/1167 - o(1)$ is attained by Rudnev and Stevens~\cite{RudSte}. It is worth pointing out that by applying the technique of Elekes using the best-known point-line incidences bound over fields of positive characteristic, due to Stevens and de~Zeeuw~\cite{SdZ}, one recovers $\varepsilon = 1/5$ as in \cite{RNRuShk}.

The exponent $\varepsilon = 1/5$ remained a threshold exponent until Shakan and Shkredov~\cite{ShaShk}, using techniques inherited from the reals, were able to break this barrier. In particular, whereas a breakthrough in progress in the reals came from the observation that bounds on $\E_3$ can be efficiently estimated using the \ST theorem (see e.g. \cite{SchShk}), Shakan and Shkredov~\cite{ShaShk} realised that the `correct' energy (with regards to the techniques currently available to us) to use for this technique over finite fields is $\E_4$. They then took advantage of the operator method (also called eigenvalue-method) introduced by Shkredov (see e.g. \cite{SchShk, ShkEnergies, ShkStructure}), which is tantamount to an ingenious double-counting argument using techniques from linear algebra.

Their result was improved by Chen, Kerr and Mohammadi~\cite{CheKerMoh} through a more efficient application of these techniques. Rudnev, Shakan and Shkredov~\cite{RudShaShk} further advanced the record by developing a new double-counting argument, which remains present in this paper, to yield the current state-of-the-art. This new argument circumvents the operator method, replacing it with recent tools in incidence geometry.
\begin{theorem}[Rudnev, Shakan, Shkredov~\cite{RudShaShk}]
Let $A\subseteq \mathbb{F}_p^*$. If $|A|<p^{36/67}$ then 
\[
\max\{|A+A|,|AA|\}\gtrsim |A|^\frac{11}{9}\,.
\]
\end{theorem}

Stimulated by their techniques, we improve this to the following result:
\begin{theorem}
\label{thm:main}
Let $\mathbb{F}$ be a field of characteristic $p \neq 2$.
Let $A\subseteq \mathbb{F}$. If $p >0$ suppose in addition that $|A|\ll p^\frac{1}{2}$. Then 
\[
\max\{|A\pm A|, |A * A|\} \gtrsim |A|^\frac{5}{4} 
\]
where $*\in \{\times, \div\}$. Moreover, this result applies to all four choices of binary operator.
\end{theorem}

We note that this result represents an improvement of $1/36$ compared to \cite{RudShaShk}, i.e. $5/4 = 11/9 + 1/36$. 

It is likely possible to relax the $p$-constraint in the statement of Theorem~\ref{thm:main}. Certainly, for the variants involving a difference or a ratio set, at certain steps of the proof where the $p$-constraint is calculated, it is possible to use the Pl\"unnecke-Ruzsa type result of \cite[Corollary~1.5]{KaSh}, instead of Lemma~\ref{lem:PRI}, which allows for a more efficient way of bounding certain iterated sum or product sets. However, to keep the proof short and more accessible, we do not attempt to optimise this constraint.

Our approach towards Theorem~\ref{thm:main} relies on an argument introduced in \cite{RudShaShk}. By double-counting the number of solutions to a tautological equation, we derive an inequality involving second and fourth moments of certain representation functions. In \cite{RudShaShk}, these energies are bounded individually, using the point-plane incidences bound of Rudnev~\cite{Rud18} and the point-line incidences bound of Stevens and de~Zeeuw~\cite{SdZ} respectively, yielding the final estimate. Here, we proceed differently. Firstly, relying on the basic observation that the arguments of \cite{RudShaShk} do not distinguish between addition and multiplication, we obtain an inequality involving both multiplicative and additive energies. Utilising a recent regularisation technique of Rudnev, as recorded by Xue \cite{Xue}, we can efficiently bound these mixed energies. This facilitates a more optimal application of the incidence results to the double-counting argument of \cite{RudShaShk}.

\subsection*{Notation} All sets in this paper are assumed to be finite. We use the Vinogradov notation $\ll, \gg$ to suppress absolute constants (independent of $\mathbb{F}$ and all sets) and $\gtrsim, \lesssim$ to suppress constants and factors of $\log(|A|)$ (or other set which will be clear from the context). We use $X \sim Y$ to mean $X\ll Y \ll X$ and $X\approx Y$ to mean $X\lesssim Y \lesssim X$.

\section{Preliminaries}
For finite sets $A,B\subseteq \F$ we use the standard representation function notation 
\[
r_{A+B}(x):=|\{(a,b)\in A\times B\colon a+b = x\}|
\]
and its obvious extensions to e.g. $r_{AA}(x)$. 

For $k> 1$ we define the additive  and multiplicative energies of the sets $A$ and $B$ to be
\[
\E_k(A,B) = \sum_{x}r_{A-B}^k(x) \quad \text{ and} \quad \E_k^\times(A,B) = \sum_x r_{A/B}^k(x)\,;
\]
if $A = B$ we typically write $\E_k(A)$ and if $k = 2$ we omit the subscript. Observe that if $A'\subseteq A$ then $\E_k(A',B)\leq \E_k(A,B)$ for any set $B$. 

The case $k=2$ corresponds to the number of solutions $(a,a',b,b')\in A^2\times B^2$ to the equation $a+b = a'+b'$ and so the Cauchy-Schwarz inequality gives in particular the bounds
\[
|A|^4 \leq \E(A)|A+A|\quad \text{ and } \quad |A|^4 \leq \E^\times(A) |AA|\,. 
\]

In our arguments, we often refer to a \emph{dyadic pigeonholing argument} applied to e.g. $\E_k(A,B)$ (and also its multiplicative analogue). This enables us to extract a set in the support of $\E_k(A,B)$, say $D\subseteq A-B$ and a number $t\geq 1$ so that $r_{A-B}(d)\in [t,2t)$ for each $d\in D$ and $\log(|A|)|D|t^k \geq \E_k(A,B)$. To generate this set $D$, we partition $A-B$ into $\lceil\log(|A|)\rceil$ sets
\[
D_i:= \{x\in A-B: 2^i\leq r_{A-B}(x) < 2^{i+1}\}
\]
for $i = 0,\dots, \lceil\log_2(|A|)\rceil.$ Then $\sum_i 2^{ki}|D_i|<\E_k(A,B)< \sum_i2^{k(i+1)}|D_i|$ and so by the pigeonhole principle, there exists $i_0$ so that $\log_2(|A|)|D_{i_0}|(2^{i_0})^k\gg \E_k(A,B)$. We take $D = D_{i_0}$ and $t = t_{i_0}$.

We require the following Pl\"unnecke-Ruzsa type inequality, a proof for which may be found in \cite{Pet12}.
\begin{lemma}\label{lem:PRI}
Let $A$ be a finite, non-empty subset of an abelian group. Then for integers $k,l\geq 0$ 
\[
|kA - lA| \leq \frac{|A+A|^{k+l}}{|A|^{k+l-1}},
\]
where $kA$ is used to denote the $k$-fold sum set of $A$.
\end{lemma}

\subsection{Regularisation arguments}
We use the following lemma in the form recorded and proved by Xue \cite{Xue} (who in turn credits Rudnev). This lemma unifies the ad hoc regularisation techniques present in the sum-product literature, e.g. \cite{RudShaShk, Warren}; an asymmetric formulation is recorded by Stevens and Warren \cite{StWa}. Although Xue states this lemma over $\mathbb{R}$, its proof is valid over abelian groups; similarly we may take $k>0$ (see e.g. \cite{StWa}).

\begin{lemma}\label{lem:CandB}
Let $A\subseteq \F$ be finite and let $k>1$ be a real number. Then there exist sets $C \subseteq B \subseteq A$ with $|C| \gtrsim |B| \gg |A|$, and a set $S_{\tau} \subseteq B-B$ and some $\tau>0$, with the properties that
$$\E_k(B) \approx |S_{\tau}|\tau^k\,,$$
$$r_{S_{\tau} + B}(c) \approx \frac{|S_{\tau}|\tau}{|A|} \quad\quad \forall c \in C.$$
\end{lemma}

We also need the following lemma, recorded by Rudnev and Stevens \cite[Lemma 1]{RudSte}; ad hoc statements of this result are similarly present within the literature, see for instance \cite[Lemma 3.1]{RudShaShk}.

\begin{lemma}\label{lem:regu}
Let $\mathcal{R}_\epsilon$ be a deterministic rule (procedure) with parameter $\epsilon\in (0,1)$ that, to every sufficiently large finite additive set $X$, associates a subset $\mathcal{R}_\epsilon(X)\subseteq X$ of cardinality $|\mathcal{R}_\epsilon(X)|\geq (1-\epsilon)|X|$.

For any such rule $\mathcal{R}_\epsilon$, any $s>1$ and a sufficiently large finite set $A$,  set $\epsilon= c_1\log^{-1}(|A|)$ for some  $c_1\in (0,1)$. 
Then there exists a set $B\subseteq A$ (depending on $\mathcal{R}_\epsilon, \,s$), with $|B|\geq (1-c_1) |A|$ such that 
\[\E_s(\mathcal{R}_\epsilon(B))\geq c_2\,\E_s(B)\,,\]
 for some constant  $c_2=c_2(s, c_1)$ in $(0,1]$. \end{lemma}

\subsection{Energy Bounds I}
In both this subsection and the subsequent we use Lemma~\ref{lem:CandB} to obtain mixed energy bounds. We first obtain bounds bound $\E_4$ and $\E_2$.

From the regularisation technique of Lemma~\ref{lem:CandB}, we obtain a subset $C\subseteq A$ for which we have the multiplicative structure described in the previous sentence. This enables us to attain the following mixed-energy bounds. 

\begin{lemma}\label{lem:E4+E2*}
Let $A\subseteq \mathbb{F}$. Then there exist sets $C\subseteq B\subseteq A$ with $|C|\gtrsim |B| \gg |A|$ so that for any set $U$ satisfying $|U||A||A-A|\ll p^2$ we have
\begin{equation}\label{e:E4+E2*}
\E_4(B)\E^\times(C,U)^2 \lesssim |A|^7 |U|^3\,.    
\end{equation}
\end{lemma}
Similarly we have the multiplicative analogue of this:
\begin{lemma}\label{lem:E4*E2+}
Let $A\subseteq \mathbb{F}$. Then there exist sets $C\subseteq B\subseteq A$ with $|C|\gtrsim |B| \gg |A|$ so that for any set $U$ satisfying  $|U||A||A/A|\ll p^2$ we have
\begin{equation}\label{e:E4*E2+}
\E_4^\times(B)\E(C,U)^2 \lesssim |A|^7 |U|^3\,.    
\end{equation}
\end{lemma}

The proofs are almost identical so we prove only the first lemma.
For this we require the following auxiliary result of Koh, Mirzaei, Pham and Shen \cite[Lemma 2.4]{KohMirPhaShe}.
\begin{lemma}\label{lem:KMPS}
Let $\mathbb{F}$ be a field of characteristic not equal to two and define $f(x,y,z) = x(y+z)$. 
Let $X,Y,Z \subseteq \mathbb{F}^*$. If $\text{char}(\mathbb{F})= p>0$, suppose that $|X||Y||Z| \ll p^2$. Then
\begin{align*}
|\{(x_1,x_2,y_1,y_2,z_1,z_2) \in X^2 \times Y^2 \times Z^2 &:   f(x_1,y_1,z_1) =   f(x_2,y_2,z_2)\}|\\
&\ll (|X||Y||Z|)^{3/2} + \max\{|X|, \min\{|Y|,|Z|\}\}|X||Y||Z|\,.
\end{align*}
\end{lemma}

We note that Koh et al. actually prove a more general statement than this version, allowing $f$ to be any `non-degenerate' quadratic polynomial.  

\begin{proof}[Proof of Lemma~\ref{lem:E4+E2*}]
Without loss of generality, we assume $0\notin A$.
We apply Lemma~\ref{lem:CandB} choosing $k=4$ to obtain sets $C\subseteq B \subseteq A$ with $|C|\gtrsim |B|\gg |A|$.
By a dyadic pigeonholing argument, we assume that $\E_4(B)\approx |D|t^4$, and from Lemma~\ref{lem:CandB}, we have
\[
r_{D+B}(c)\approx \frac{|D|t}{|A|} ~\quad \forall c \in C\,.
\]

Consider now $\E^\times(C,U)$. Let $U' = U \setminus\{0\}$ and $D' = D \setminus\{0\}$. We have 
\begin{align*}
\E^\times(C,U) 
&= |\{(c_1,c_2,u_1,u_2) \in C^2\times U'^2: c_1u_1 = c_2u_2\}| + |C|^2
\\
&\lesssim \frac{|A|^2}{|D|^2t^2}|\{(b_1,b_2,d_1,d_2u_1,u_2) \in B^2\times D^2\times U'^2: (d_1 + b_1)u_1 = (d_2 + b_2)u_2\}|\\
&\leq \frac{|A|^2}{|D|^2t^2}\left(|\{(b_1,b_2,d_1,d_2u_1,u_2) \in B^2\times D'^2\times U'^2: (d_1 + b_1)u_1 = (d_2 + b_2)u_2\}| \right.\\
&\left.\quad +
2\frac{|B|^2|U|^2|D|}{\max\{|B|,|U|,|D|\}}
+ |B||U|\min\{|B|,|U|\}\right)
\\
&\ll \frac{|A|^2}{|D|^2t^2}\left( |D|^{3/2}|B|^{3/2}|U|^{3/2} + \max\{|U|,\min\{|D|,|B\}\}) |U||D||B| 
+ \frac{|D||B|^2|U|^2}{\max\{|D|,|B|,|U|\}} \right)
\end{align*}
where the final sum is to account for the possibility that  $0 \in D$.
A case analysis shows that the final term is always smaller than the second term 
and so 
\begin{align*}
\E^\times(C,U) &\lesssim\frac{|A|^2}{|D|^2t^2}\left( |D|^{3/2}|B|^{3/2}|U|^{3/2} + \max\{|U|,\min\{|D|,|B|\}\} |U||D||B| \right)\\
&:=  \frac{|A|^2}{|D|^2t^2}\left( |D|^{3/2}|B|^{3/2}|U|^{3/2} + M |U||D||B| \right) \,. 
\end{align*}
We claim that $|D|^{3/2}|B|^{3/2}|U|^{3/2} > M |U||D||B|$ to complete the proof. Indeed, if this is not the case, then we will show that either we obtain a contradiction, or else we are done by using the trivial estimate: $\E_4(B)\E^\times(C,U)^2 \leq |D||B|^4|C|^2|U|^2\min\{|C|,|U|\}^2 $.

{\bf{Case 1: $M = |U|$:}} Then  $|D|^{3/2}|B|^{3/2}|U|^{3/2} < M |U||D||B|$ implies that $|U|> |D||B|$ and so using the trivial estimate we have 
\[
\E_4(B)\E^\times(C,U)^2 \leq |D||B|^8|U|^2 < |B|^7|U|^3\,.\]

{\bf{Case 2: $M = |B|$:}} This can only happen if $|D|>|A|>|U|$.
Then  $|D|^{3/2}|B|^{3/2}|U|^{3/2} < M |U||D||B|$ implies that $|B|> |D||U|$ and so using the trivial estimate we have 
\[
\E_4(B)\E^\times(C,U)^2 \leq |D||B|^6|U|^4 < |B|^7|U|^3\,.\]

{\bf{Case 3: $M = |D|$:}} This can only happen if $|B|>|U|>|D|$. Then  $|D|^{3/2}|B|^{3/2}|U|^{3/2} < M |U||D||B|$ implies that $|D|> |B||U|$. On the other hand, $|B||U|>|D|$ and so we reach a contradiction.

Finally, we justify our application of Lemma~\ref{lem:KMPS}. This follows from $|A||D||U|\leq |A||A-A||U|\ll p^2$.
\end{proof}
\subsection{Energy Bounds II} 
We recall \cite[Theorem~2.1]{RudShaShk}, which is a consequence of a point-line incidence bound of Stevens and de~Zeeuw.
\begin{lemma}\label{lem:SdZCP}
Let $A, B, C, D\subset \F_p$. If $|A||B||C||D|^2\ll p^4$, then
\[
|\{(a, b, c, d)\in A\times B\times C\times D: c= ab+d\}|\ll (|A||B||C|)^{3/4}|D|^{1/2} +|A||D| + |B||C|.
\]
\end{lemma}

\begin{lemma}\label{lem:E4*E4+}
Let $A, U\subset \F_p$. There exist $C\subset B\subset A$, with $|C|\gtrsim |B|\gg |A|$ such that, assuming $|A/A||A||A - U||U|^2\ll p^4$,
\[
\E^\times_4(B) \E_4(C,U)\lesssim |A|^7 |U|^2.
\]
\end{lemma}
\begin{proof}

We begin by applying Lemma~\ref{lem:CandB} to $A$ in its multiplicative form to obtain sets $C\subseteq B \subseteq A$ so that we have $\E_4^\times(B) \sim |S_\tau|\tau^4$ and  $r_{S_\tau B}(c) \approx |S_\tau|\tau|A|^{-1}$ for each $c\in C$. Note that $\tau \leq |A|$.

By a dyadic pigeonhole argument, we extract $D_t\subseteq C-U$ so that for some $1\leq t \leq \min\{|C|,|U|\}$ we have $\E_4(C,U) \sim |D_t|t^4$.

By Lemma~\ref{lem:SdZCP} we have
\begin{align*}
    |D_t|t &\leq |\{(d,c,u)\in D_t\times C\times U \colon d = c -u\}|\\
    &\approx     \frac{|A| }{|S_\tau|\tau}|\{(d,s,b,u)\in D_t \times S_\tau\times B\times U \colon  d =sb - u\}|\\
    &\ll  \frac{|A| }{|S_\tau|\tau} \left( (|D_t||B||S_\tau|)^{3/4}|U|^{1/2} +
    |S_\tau||U| + |B||D_t|
    \right)\,.
\end{align*}

If the first term dominates then rearranging yields the desired bound. 

Suppose instead that the second term dominates so that 
$|D_t|t ~\tau \lesssim |A||U|$.
Then 
\[
\E_4^\times(B)\E_4(C,U) \sim |S_\tau|\tau^4 |D_t|t^4  =
(|S_\tau|\tau) \tau^2 (\tau ~|D_t|t) t^3
\lesssim (|B|^2) |B|^2 (|A||U|)(|A|^2|U|) \leq |A|^7 |U|^2\,.
\]

Now suppose that the final term dominates, so that $|S_\tau|\tau~ t \lesssim |A|^2$.
Then using that $t\leq \min\{|A|,|U|\}$ we have
\[
\E_4^\times(B)\E_4(C,U) \sim |S_\tau|\tau^4 |D_t|t^4  
= \tau^3 (|S_\tau|\tau ~t) (|D_t|t) t^2 
\lesssim |U|^2 |A| |A|^2 |B|^2 |B|^2 \leq |U|^2 |A|^7 \,.
\] 
It remains to justify the $p$-constraint for the application of Lemma~\ref{lem:SdZCP}. \\
We require that $|S_\tau||D_\tau||B||U|^2 \ll p^4$. Since $S_\tau \subseteq A/A$ and $D_\tau \subseteq A-U$, the hypothesis $|A/A||A-U||A||U|^2 \ll p^4$ renders this application valid. 
\end{proof}

We also record the converse analogue of Lemma~\ref{lem:E4*E4+} whose proof follows almost identically, swapping each instance of addition with multiplication. 
\begin{lemma}\label{lem:E4+E4*}
Let $A, U\subset \F_p$. There exist $C\subset B\subset A$, with $|C|\gtrsim |B|\gg |A|$ such that, assuming $|A-A||A||A/U||U|^2\ll p^4$,
\[
\E_4(B)\E_4^\times(C,U)\lesssim |A|^7 |U|^2.
\]
\end{lemma}

\section{Arguments of Rudnev, Shakan and Shkredov}
We extract the following proposition and proof from the arguments of Rudnev, Shakan and Shkredov \cite{RudShaShk}. 
\begin{proposition}\label{prop:RSS argument additive}
Given $A\subset \F_p$, there exists a set $B\subseteq A$ with $|B|\gg |A|$ so that 
\[
\E_{4/3}(B)^3 \lesssim \frac{|A+A|^8 \E_4(A)^2 \E_4(A,\mathcal{E})\mu^4 \nu^4 }{|A|^{24}}
\]
where $\E_{4/3}(B)\approx |\mathcal{F}|\nu^{4/3}$ for a set $\mathcal{F}\subseteq B-B$ and a number $\nu\geq 1$ so that $r_{B-B}(f)\in [\nu,2\nu)$ for all $f\in \mathcal{F}$. Moreover, there exists $\mathcal{E} \subseteq A-\mathcal{F}$ and $\mu\geq 1$ so that $\E(A,\mathcal{F})\approx |\mathcal{E}|\mu^2$ for  so that $r_{A- \mathcal{F}}(e) \in[\mu,2\mu)$, for all $e\in \mathcal{E}$.
\end{proposition}
\begin{proof}
We first apply Lemma~\ref{lem:regu} to the set $A$ choosing $s = 4/3$ and using the rule
\[
\mathcal{R}_\epsilon(A):= \{a\in A: |\{b\in A: a+b \in P_A\}| \geq \frac23|A|\}
\]
where $P_A:= \{x \in A+A: r_{A+A}(x) \geq \epsilon\frac{|A|^2}{|A+A|}\}$. This rule refines $A$ according to popular sums and is admissible for Lemma~\ref{lem:regu}. (We could replace this rule with an analogous procedure which refines $A$ according to popular \emph{differences}, which would replace all sum sets with difference sets in the subsequent arguments).

Let $C:= \mathcal{R}_\epsilon(B)$ be the set obtained from $A$ using Lemma~\ref{lem:regu}. Suppose that  $\E_{4/3}(C) \approx |D|t^{4/3}$ by a dyadic pigeonhole argument. Note that 
\begin{equation}\label{eqn:E43CB}
  \E_{4/3}(C)\gg \E_{4/3}(B).  
\end{equation}

Now let us count solutions $(a,b,c,d)\in B^4$ to the following equation 
\begin{equation}\label{e:tautological}
a-b = (a + c ) - (b+c) = (a+d) - (b+d)
\end{equation}
where $a-b\in D$, and $a+c,b+c,a+d,b+d \in P_C$. 

A consequence of our regularisation ensures that we have at least 
$|D|t (2|B|/3)^2 \sim |D|t|A|^2$ solutions. 

On the other hand, using a now-standard technique of counting the number of solutions via equivalence classes (see its origins in \cite{RudShaShk} and its direct analogue over the reals in \cite{RudSte}), we get the upper bound 
\[
\#\{\text{solutions}\}\leq \sqrt{\E_4(B)}\sqrt{|\{(x_1,y_1,x_2,y_2,d)\in P_{C}\times D: d = x_1 - y_1 = x_2 - y_2 \}|}\,.
\]

Clearly $\E_4(B) \leq \E_4(A)$. We then combine the lower and upper bounds for the number of solutions to the tautological equation \eqref{e:tautological} (raised to the power four) and use the popularity of the set $P_C$ to obtain 
\begin{align*}
    |D|^4t^4|A|^8 &\ll \E_4(A)^2|\{(x_1,y_1,x_2,y_2,d)\in P_{C}\times D: d = x_1 - y_1 = x_2 - y_2 \}|^2\\
    &\lesssim \E_4(A)^2 |A+A|^{8}|A|^{-16}\\
    &\quad \cdot
    |\{(a_1,a_2,a_3,a_4,a_5,a_6,a_7,a_8,d)\in B^8 \times D: d =a_1 + a_2 - a_3 - a_4 = a_5+a_6 - a_7 - a_8 \}|^2\\
    &\approx \frac{\E_4(A)^2 |A+A|^{8}}{|A|^{16}} \nu^4
    |\{(a_1,a_2,a_3,a_4,f_1,f_2,d)\in B^4\times \mathcal{F}^2 \times D: d =a_1 + f_1 - a_2 = a_3+f_2 -a_4 \}|^2
\end{align*}
where $\mathcal{F}\subseteq B - B$, $r_{A-A}(f)\in [\nu,2\nu)$ for all $f\in \mathcal{F}$ and $\E_{4/3}(B) \approx |\mathcal{F}|\nu^{4/3}$.

We again dyadically localise, to a set $\mathcal{E} \subseteq A-\mathcal{F}$ so that $r_{A- \mathcal{F}}(e) \in[\mu,2\mu)$, for all $e\in \mathcal{E}$ and $\E(A,\mathcal{F})\approx |\mathcal{E}|\mu^2$.
Thus
\begin{align*}
    |D|^4t^4|A|^8 & \lesssim   \frac{\E_4(A)^2 |A+A|^{8}}{|A|^{16}} \nu^4\mu^4
    |\{(a_1,a_2, e_1,e_2,d)\in B^2\times \mathcal{E}^2 \times D: d =a_1 - e_1 = a_2 - e_2 \}|^2\\
    &=  \frac{\E_4(A)^2 |A+A|^{8}}{|A|^{16}} \nu^4\mu^4 \left(\sum_{d\in D}r_{A-\mathcal{E}}(d)^2\right)^2 \\
    &\leq   \frac{\E_4(A)^2 |A+A|^{8}}{|A|^{16}} \nu^4\mu^4 |D|~\E_4(A,\mathcal{E})\,.
\end{align*}

By rearranging and noting that $\E_{4/3}(C)^3\approx  |D|^3t^4$, we obtain the required inequality.

\end{proof}

We record that we have a multiplicative analogue of Proposition~\ref{prop:RSS argument additive}. The proof is almost identical, and involves merely swapping all instances of addition and multiplication. We can also swap all instances of the product set $AA$ with the ratio set $A/A$. 
\begin{proposition}\label{prop:RSS argument multiplicative}
Let $A\subset \F_p$. There exists a set $B\subseteq A$ with $|B|\gg |A|$ so that 
\[
\E^\times_{4/3}(B)^3 \lesssim \frac{|AA|^8 \E^\times_4(A)^2 \E^\times_4(A,\mathcal{E})\mu^4 \nu^4 }{|A|^{24}}
\]
where $\E_{4/3}^\times(B)\approx |\mathcal{F}|\nu^{4/3}$ for some $\mathcal{F}\subseteq B/B$ and $\nu\geq 1$ so that $r_{B/B}(f)\in [\nu,2\nu)$ for all $f\in \mathcal{F}$. Moreover, there exist $\mathcal{E} \subseteq A/\mathcal{F}$ and $\mu\geq 1$ so that $\E^\times(A,\mathcal{F})\approx |\mathcal{E}|\mu^2$ and $r_{A/ \mathcal{F}}(e) \in[\mu,2\mu)$, for all $e\in \mathcal{E}$.
\end{proposition}

\section{Proof of Main Theorem}
We prove only the most-studied version of Theorem~\ref{thm:main} of sums and products; the other variants are deduced in an almost identical manner. 

\subsection{Refinement}\label{sec:refinement}
We begin with four applications of Lemma~\ref{lem:CandB} (albeit within Lemmas~\ref{lem:E4+E2*}, \ref{lem:E4*E2+}, \ref{lem:E4*E4+} and \ref{lem:E4+E4*}).

Firstly, from Lemma~\ref{lem:E4+E2*} applied to the set $A$ we obtain sets $A_2\subseteq A_1 \subseteq A$ so that 
\[
\E_4(A_1)\E^\times(A_2, U)^2\lesssim |A|^7 |U|^3~~~~\text{for any } U\,.
\]

Secondly, we apply Lemma~\ref{lem:E4*E2+} to the set $A_2$ to get $A_4\subseteq A_3\subseteq A_2$ with 
\[
\E_4^\times(A_3)\E(A_4, U)^2\lesssim|A|^7 |U|^3~~~~\text{for any } U\,.
\]

We now continue refining our set in order to take advantage of  Lemmas~\ref{lem:E4*E4+} and \ref{lem:E4+E4*}.
Let us apply  Lemma~\ref{lem:E4+E4*} to the set $A_4$ to obtain $A_6\subseteq A_5 \subseteq A_4$ so that for any set $U$ we have
\[
\E_4(A_5)\E_4^\times(A_6,U)\lesssim |A|^7 |U|^2.
\]
Finally, we apply Lemma~\ref{lem:E4*E4+} to the set $A_6$ to obtain 
$A_8\subseteq A_7 \subseteq A_6$ so that for any set $U$ we have
\[
\E^\times_4(A_7) \E_4(A_8,U)\lesssim |A|^7 |U|^2.
\]
Note that in each refinement stage we retain a positive proportion of the set, so that $|A_8|\gtrsim |A|$.

\subsection{The calculation}
We now apply Propositions~\ref{prop:RSS argument additive} and \ref{prop:RSS argument multiplicative} to the set $A_8$. We multiply the ensuing bounds. 

To summarise, we obtain subsets $ B_1\subseteq A_8$ and $ B_2\subseteq A_8$ with $|B_1|,|B_2,|\gtrsim |A|$  so that
\[
\E_{4/3}(B_1)^3~\E_{4/3}^{\times}(B_2)^3\lesssim \E_4(A_8)^2~\E_4^\times(A_8)^2 \frac{|A+A|^{8}|AA|^{8}}{|A|^{48}} \E_4(A_8,\mathcal{E}_1)\mu_1^4 \nu_1^4 \E^{\times}_4(A_8,\mathcal{E}_2) \mu_2^4\nu_2^4 
\]
where
\begin{enumerate}[label=(\roman*)]
    \item $\mathcal{F}_1 \subseteq B_1 - B_1$ and $\mathcal{F}_2\subseteq B_2 - B_2$.
    \item $\E_{4/3}(B_1)\approx |\mathcal{F}_1|\nu_2^{4/3}$ and 
    $\E_{4/3}^\times(B_2)\approx |\mathcal{F}_2|\nu_2^{4/3}$.
    \item $\mathcal{E}_1 \subseteq A_8- \mathcal{F}_1$ and $\mathcal{E}_2 \subseteq A_8 / \mathcal{F}_2$
    \item $\E(A_8,\mathcal{F}_1)\approx |\mathcal{E}_1|\mu_1^2$ and 
$\E^\times(A_8,\mathcal{F}_2)\approx |\mathcal{E}_2|\mu_2^2$.
\end{enumerate}

In the subsequent analysis, we will make ample use of bounds of the type $\E_4(A_8,U) \leq \E_4(A_7,U)$ etc to enable us to take advantage of Section~\ref{sec:refinement}. 

From Lemmas~\ref{lem:E4*E4+} and \ref{lem:E4+E4*}, we have
\begin{align*}
\E_{4/3}(B_1)^3\E_{4/3}^{\times}(B_2)^3&\lesssim \E_4(A_8)\E_4^\times(A_8)\frac{|A+A|^{8}|AA|^{8}}{|A|^{34}} |\mathcal{E}_1|^2\mu_1^4 \nu_1^4 | \mathcal{E}_2|^2\mu_2^4 \nu_2^4 \\
&\approx \E_4(A_8)\E_4^\times(A_8)\frac{|A+A|^{8}|AA|^{8}}{|A|^{34}} \E(A_8, \mathcal{F}_1)^2\E^{\times}(A_8, \mathcal{F}_2)^2\nu_1^4\nu_2^4\,.
\end{align*}
Recalling $\E_{4/3}(B_1) \approx |\mathcal{F}_1|\nu_1^{4/3}$ and $\E_{4/3}^\times(B_2) \approx |\mathcal{F}_2|\nu_2^{4/3}$, we use Lemmas~\ref{lem:E4+E2*} and \ref{lem:E4*E2+} to get
\begin{align*}
\E_{4/3}(B_1)^3\E_{4/3}^{\times}(B_2)^3 &\ll \frac{|A+A|^{8}|AA|^{8}}{|A|^{20}}|\mathcal{F}_1|^3\nu_1^4|\mathcal{F}_2|^3\nu_2^4\\
&\ll \frac{|A+A|^{8}|AA|^{8}}{|A|^{20}}\E_{4/3}(B_1)^3\E_{4/3}^{\times}(B_2)^3,
\end{align*}
which gives the required inequality.

\subsection{Justification of the $p$-constraint}
It remains to justify our use of Lemmas \ref{lem:E4+E2*}, \ref{lem:E4*E2+}, \ref{lem:E4*E4+} and \ref{lem:E4+E4*}.  In particular, through each application we collect the following constraints: 
\begin{enumerate}[label=(\roman*)]
\item \label{one} $|\mathcal{E}_1|^2|A||A/A||A-\mathcal{E}_1|\ll p^4$ because we apply bounds on $\E_4^\times(A_8)\E(A_8,\mathcal{E}_1)$. 
\item \label{two} $|\mathcal{E}_2|^2|A||A-A||A/\mathcal{E}_2|\ll p^4$ because we apply bounds on $\E_4(A_8)\E^\times(A_8,\mathcal{E}_2)^2$. 
\item \label{three}$|\mathcal{F}_1||A||A-A|\ll p^2 $ because we apply bounds on 
$\E_4^\times(A_8) \E(A_8,\mathcal{F}_1)^2$.
\item \label{four} $|\mathcal{F}_2||A||A/A|\ll p^2 $ because we apply bounds on $\E_4(A_8)\E_2^\times(A_8,\mathcal{F}_2)^2$
\end{enumerate}

We repeatedly use Lemma~\ref{lem:PRI} to justify each of the bounds. We note the symmetry of addition and multiplication appearing in \ref{two} and \ref{four}, and thus only
illustrate how to justify the constraints from \ref{one} and \ref{three}. 

By Lemma~\ref{lem:PRI}, we have $|\mathcal{E}_1|\leq |A+A-A|\leq |A+A|^{3}|A|^{-2}$ and $|A-\mathcal{E}_1|\leq |A+A-A-A|\leq |A+A|^{4}|A|^{-3}$.

Hence the constraint $|\mathcal{E}_1|^2|A||A/A||A-\mathcal{E}_1|\ll p^4$ is satisfied if
$|A+A|^{10}|AA|^2|A|^{-7}\ll p^4$. Suppose that this does not hold. Then, since $p>|A|^2$, we have
$|A+A|^{10}|AA|^2\gg |A|^{15}$, and so we are done. 

Let us now consider the constraint in \ref{three}. Using Lemma~\ref{lem:PRI}, we see that $|\mathcal{F}_2||A||A/A|\ll p^2 $ is satisfied if
$|A+A|^2 |AA|^2 \ll |A| p^2 $. If this is not the case, then $|A+A|^2|AA|^2 \gg |A|^4$, as required.

\section*{Acknowledgements}
The second author was supported by the Austrian Science Fund FWF grants P 30405 and P 34180. We thank Audie Warren for his helpful comments.

\end{document}